\documentclass[11pt,reqno]{amsart}
\usepackage{enumerate}
\usepackage{amsthm,amssymb}
\usepackage{amsmath}
\usepackage{graphicx,epsfig}
\usepackage{bbm}
\def\1{\mathbbm{1}}
\usepackage{amssymb}
\usepackage{latexsym}

\usepackage[mathscr]{euscript}
\usepackage{mathrsfs}
\usepackage[usenames]{color}
\usepackage{hyperref}
 \usepackage{soul}

\def\CC{{\mathbb{C}}}
\def\A{{\mathbb{A}}}
\def\R{{\mathbb{R}}}

\def\B{\mathbb{B}}
\def\X{\mathbb{X}}

\newtheorem{theo}{Theorem}[section]

\newtheorem{asp}[theo]{Assumption}

\theoremstyle{definition}
\newtheorem{exa}[theo]{Example}
\newtheorem{rem}[theo]{Remark}
\newtheorem{defi}[theo]{Definition}

\def\al{\alpha}
\def\om{\omega}

\def\ga{\gamma}
\def\si{\sigma}

\def\la{\lambda}

\def\t{\tau}

\def\calL{{\mathcal{L}}}

\def\calU{{\mathcal{U}}}

\def\calT{{\mathcal{T}}}

\def\calB{{\mathcal{B}}}

\def\calA{{\mathcal{A}}}

\def\C{\mathbb C}

\def\F{\mathnormal F}

\def\T{\mathbb T}

\newcommand{\mscr}{\mathnormal}

\begin{document}

\numberwithin{equation}{section}

\title[]{Abstract boundary delay systems and application to network flow}
\author[]{Andr\'as B\'atkai, Marjeta Kramar Fijav\v{z}, and  Abdelaziz Rhandi}

\address{A. B\'atkai: P\"adagogische Hochschule Vorarlberg,
Liechtensteinerstrasse 33-37, A-6800 Feldkirch, Austria;  andras.batkai@ph-vorarlberg.ac.at}
\address{M. Kramar Fijav\v{z}: University of Ljubljana, Faculty of Civil and Geodetic Engineering,
Jamova 2, SI-1000 Ljubljana, Slovenia / Institute of Mathematics, Physics, and
Mechanics, Jadranska 19, SI-1000 Ljubljana, Slovenia; marjeta.kramar@fgg.uni-lj.si}
\address{A. Rhandi: Dipartimento di Matematica, Universit\`a degli Studi di Salerno, Via Giovanni Paolo II, 132, I-84084 Fisciano (Sa), Italy; arhandi@unisa.it}
\thanks{}
\vskip 0.5cm
 \keywords{Operator semigroup, well-posed linear system, perturbation of a generator, transport equation with delay, network flow, metric graph}

\medskip
\medskip

\bigskip
\begin{abstract}
This paper investigates the well-posedness and positivity of solutions to a class of delayed transport equations on a network.  The material flow is delayed at the vertices and along the edges.  The problem is reformulated as an abstract boundary delay equation, and well-posedness is proved by using the Staffans-Weiss theory.  We also establish spectral theory for the associated delay operators and provide conditions for the positivity of the semigroup.
\end{abstract}

\maketitle
\section{Introduction}

Many systems are subject to time delays, which can affect their well-posedness and stability. Delays can occur within the differential equations themselves, within the boundary conditions, or even in both (e.g., \cite{BP,BDDM,Had,HMR,HL-Book}). This paper will focus on the latter scenario, where delays occur also within the boundary conditions.

Consider a delayed transport process taking place along the edges of a metric graph.
We start by a finite simple directed graph 
 with set of vertices
$V = \{v_1,\dots,v_n\}$ and set of directed edges, $E = \{{e}_1,\dots ,{e}_m\}$ and obtain a metric graph by parametrizing each edge ${e}_j$ as: $\bar{e}_j\colon [0,1]\to {e}_j$, $j=1,\dots,m$. For simplicity, we denote the vertices at the endpoints of the edge ${e}_j$ by  $\bar{e}_j(0)$ and  $\bar{e}_j(1)$, respectively.
We shall assume that in every vertex, there is at least one outgoing as
well as at least one incoming edge; that is, there are no sources nor sinks.

To describe the structure of the graph, we use the \emph{weighted transposed adjacency matrix of the line graph} $\mathbb{B} = (\om_{ij})_{m\times m}$ where $\om_{ij} \ge 0$ are the weights on the edges of the original graph such that
\begin{equation}\label{eqn:adjMat}
\om_{ij} \ne 0 \iff \bar{e}_j(0)=\bar{e}_i(1).
\end{equation}
Thus, the nonzero entries of $\mathbb{B}$ correspond to the adjacent edges, respecting the chosen orientation.
We further assume
\begin{align}\label{eq:w}
\sum_{i=1}^m \om_{ij}=1\qquad\text{for all}\;j\in \{1,\dots,m\},
\end{align}
making matrix $\mathbb{B}$  column stochastic.
For some further properties of this matrix and some other graph matrices, we refer to \cite[Sect.~18]{Positive2017}.

Let $u_j(x,t)$ denote the
distribution of material transported along edge $
{e}_j$, depending on time $t$ and  location
$x \in [0,1]$.
We shall assume that along each edge, the material flows from $1$ to $0$ and is possibly delayed in the vertices as well as along the edges. In each vertex, the incoming material is distributed
into the outgoing edges according to the given weights $\om_{ij}$. Note that by \eqref{eq:w}, the mass is conserved.
The model can be described by the following system.
\begin{align}\label{del-f}
\begin{cases} \partial_t u_j(x,t)=\partial_x u_j(x,t)+ \left(\int^0_{-1} d\mu_j(\theta)u_j(\cdot,\theta+t)\right)(x),
& t\ge 0,\; x\in (0,1), \cr
u_j(1,t)={\sum_{k=1}^m \mathbb{B}_{jk}\left(u_k(0,t)+\int^0_{-1} d\nu_k(\theta) u_k(\cdot,\theta+t)\right)},& t\ge 0,\cr
u_j(x,0)=f_j(x),& x\in (0,1), \cr
u_j(x,\t)=g_j(x,\t),& x\in (0,1),\; \t\in [-1,0],
\end{cases}
\end{align}
for  $j\in \{1,\dots,m\}$.
Here,
\[\mu_{k}:[-1,0]\to \calL(L^p([0,1],\C))_+\quad\text{and}\quad \nu_k:[-1,0]\to \calL(L^p([0,1],\C),\C)_+\]
 are positive operator-valued functions of bounded variation that are continuous and vanishing at zero for $k=1,\dots,m$.
Notice that the delay terms are nonlocal both on the edges as well as in the boundary.

To investigate the properties of the solutions of this system, we have to make a theoretical detour and investigate in Sections \ref{sec:2}, \ref{sec:3}, and \ref{sec:4} the properties of the  abstract boundary delay equation given by:
\begin{align} \label{BED}
\begin{cases}
\dot{z}(t) = A_m z(t) + P z_t, & t \geq 0, \cr
G z(t) = M z(t) + L z_t, & t \geq 0, \cr
z(0) = x,\;z(s) = \varphi(s) & s \in [-1,0].
\end{cases}
\end{align}
Here, the operator $A_m: Z \subset X \to X$ represents a closed operator on the Banach space $X$, where $Z$ is continuously and densely embedded in $X$. The linear operators $G$ and $M: Z \to U$ (with $U$ being a boundary Banach space) are involved in the boundary conditions, and there are two delay operators
\begin{align}\label{eq:opPL}
P f := \int^0_{-1} d\mu(\theta) f(\theta)\quad \text{ and }\quad L f := \int^0_{-1} d\nu(\theta) f(\theta),
\end{align}
$f \in W^{1,p}([-1,0],X)$, where $\mu:[-1,0]\to\mathcal{L}(X)$ and $\nu:[-1,0]\to\mathcal{L}(X,U)$ are functions of bounded variation, continuous at zero, with $\mu(0)=\nu(0)=0$. The initial conditions are given by $x\in X$ and $\varphi\in L^p([-1,0],X)$. Additionally, for any $t \geq 0$, the history function $z_t(\cdot):=z(\cdot+t):[-1,0]\to X$ is defined as $z_t(\theta)=z(t+\theta)$ for any $\theta \in [-1,0]$.

We will prove the well-posedness of Equation \ref{del-f} in Theorem \ref{thm-network}.

The abstract equation mentioned above has been extensively studied in the literature in the case of $M=0$ and $L=0$, as evidenced in \cite{BP,Had}. In these works, the authors employ product spaces and matrix operators to reframe the delay equation as a Cauchy problem. Subsequently, they utilize the Miyadera-Voigt perturbation theorem to establish the well-posedness of the Cauchy problem. Additionally, \cite{Had} incorporates admissible observation operators and their Yosida extension to provide a representation of the solutions.

The problem \eqref{BED} with $P=L=0$ was first studied by Greiner \cite{Gre} and Salamon \cite{Sal2} in the case where $M\in \mathcal{L}(X,U)$, and in \cite{ABE,HMR} in the unbounded case.
The investigation of the aforementioned equations under the conditions of $P=0$ and $M=0$ has been conducted in \cite{ABE,Had-Lah,HMR}. In this particular scenario, the authors have employed the Staffans-Weiss perturbation theorem \cite{Staf} and \cite{Weis5} to establish the well-posedness of the equations, see also \cite{Cu-Zw,TW}.

The network transport process with delays taking place only in the vertices (that is, in the case of $P=0$) has already been considered in \cite{BDR}. 
We address the well-posedness of equation \eqref{BED} in the general case.  To reach this goal we employ the product space $\X=X\times L^p([-1,0],X)$ and transform the boundary delay equation \eqref{BED} into a Cauchy problem governed by a matrix operator $\calA_{P,L}:D(\calA_{P,L})\subset \X\to \X$,  defined by equation \eqref{Big-generator}. By introducing appropriate conditions in Assumption \ref{tildeA}, and utilizing an abstract perturbation result proved in \cite[Theorem 4.1]{HMR}, we establish that the operator $\calA_{P,L}$ coincides with the generator of a closed-loop system, thereby generating a strongly continuous semigroup $(\calT_{P,L}(t))_{t\ge 0}$ on $\X$ that satisfies the equation $$ \calT_{P,L}(t)\left(\begin{smallmatrix}x\\ \varphi\end{smallmatrix} \right)=\left(\begin{smallmatrix}z(t)\\ z_t\end{smallmatrix} \right)$$ for all $t\ge 0$ and $\left(\begin{smallmatrix}x\\ \varphi\end{smallmatrix} \right)\in \X$. Furthermore, we provide details regarding the spectrum of the operator $\calA_{P,L}$ and present a precise formulation of its resolvent operator. This facilitates the demonstration of the positivity of the semigroup $(\calT_{P,L}(t))_{t\ge 0}$. For the theory of strongly continuous semigroups and positivity preserving semigroups we refer the reader to the more recent monographs \cite{EN-NA} and \cite{Positive2017}.

The paper is organized as follows: Section \ref{sec:2} provides a concise summary of infinite-dimensional closed-loop systems. Section \ref{sec:3} is dedicated to proving the well-posedness of the abstract delay equation \eqref{BED} through the Cauchy problem governed by the operator $\calA_{P,L}$. In Section \ref{sec:4}, we calculate the spectrum of the operator $\calA_{P,L}$, its resolvent operator, and obtain the positivity of the semigroup under appropriate conditions. The final section applies the abstract results to study network flow equations with delays given in \eqref{del-f}.

\subsection*{Notation}
Let $X$ be a Banach space with norm $\|\cdot\|$ and $A:D(A)\subset X\to X$ be the generator of a strongly continuous semigroup $(T(t))_{t\ge 0}$ on $X$. We denote by $X_1$ the domain $D(A)$ endowed with the graph norm $\|x\|_1=\|x\|+\|Ax\|$ for $x\in X$, which is a Banach space. By $\rho(A)$  we denote the resolvent set of $A$ and by $R(\al,A):=(\al-A)^{-1},$ $\al\in\rho(A),$ the resolvent operator of $A$. 

We denote by $X_{-1}$ the completion of $X$ with respect to the norm $\|x\|_{-1}:=\|R(\al,A)x\|$ for $x\in X$ and some $\al\in\rho(A)$, so  $X_{-1}$ is a Banach space satisfying
\begin{align*}
X_1\subset X\subset X_{-1},
\end{align*}
densely and with continuous embedding. We mention that if the state space $X$ is reflexive, then the extrapolation space $X_{-1}$ is isomorphic to the topological dual $D(A^\ast)',$ where $A^\ast$ is the adjoint operator of the generator $A$.
Note, that the semigroup $(T(t))_{t\ge 0}$ on $X$ extends to a strongly continuous semigroup $(T_{-1}(t))_{t\ge 0}$ on $X_{-1}$ whose generator $A_{-1}:X\to X_{-1}$ is the extension of $A$ to $X_{-1}$.

We denote by $\calL(Z_1,Z_2)$ the Banach algebra of all linear bounded operators from a Banach space $Z_1$ to a Banach space $Z_2$.

\section{Infinite-dimensional closed-loop systems}\label{sec:2}
Here, we present the unbounded control systems approach to boundary problems. We recall the necessary notation and terminology to formulate the main generation result we will need, Theorem \ref{HMR15}. We mostly rely on the theory of unbounded perturbations of a generator as developed in \cite{HMR}.
{Let us note that a similar theory has been developed simultaneously by Adler, Bombieri, and Engel in \cite{ABE}.}

Throughout this section $\mscr{X}$, $\mscr{U}$, and $\mscr{Z}$ are Banach spaces (we use the same notation  $\|\cdot\|$ to specify the corresponding norms) such that $\mscr{Z}\subset \mscr{X}$ with dense and continuous embedding.

\begin{asp}\label{Greiner}
For closed linear operator $\mscr{A}_m\colon \mscr{Z}\to \mscr{X}$ and  linear {boundary operator} $\mscr{G}\colon \mscr{Z}\to \mscr{U}$ we assume the following.
\begin{enumerate}
  \item 
  The restricted operator $\mscr{A}\subset \mscr{A}_m$ with domain $D(\mscr{A})=\ker(\mscr{G})$ generates a $C_0$-semigroup $(\mscr{T}(t))_{t\ge 0}$ on $\mscr{X}$.
  \item 
  The boundary operator $\mscr{G}\colon \mscr{Z}\to \mscr{U}$ is surjective.
\end{enumerate}
\end{asp}

Now consider the linear system
\begin{align}\label{Boundary-Obser}
\begin{cases} \dot{x}(t)=\mscr{A}_m x(t),\quad x(0)=x^0,& t\ge 0,\cr \mscr{G}x(t)=0,& t\ge 0,\cr y(t)=\mscr{M}x(t),& t\ge 0,\end{cases}
\end{align}
where $\mscr{M}\colon \mscr{Z}\to \mscr{U}$ is a linear operator. We define the so-called \emph{observation operator} as
\begin{align}\label{operator-C}
\mscr{C}:=\mscr{M}_{|D(\mscr{A})}\colon D(\mscr{A})\to \mscr{U}.
\end{align}
According to Assumption \ref{Greiner}, if $x(\cdot)$ takes its values in $D(A)$, we can reformulate system \eqref{Boundary-Obser} as
\begin{align}\label{A-C}
\begin{cases} \dot{x}(t)=\mscr{A} x(t),\quad x(0)=x^0,& t\ge 0,\cr y(t)=\mscr{C}x(t),& t\ge 0.\end{cases}
\end{align}
Clearly, for any $t\ge 0$ and $x^0\in \mscr{X}$ we obtain the solution of \eqref{A-C} as $x(t)=\mscr{T}(t)x^0$. Thus, \begin{align*} y(t):=y(t;x^0)=\mscr{C}\mscr{T}(t)x^0,\quad t\ge 0,\quad x^0\in D(\mscr{A}).\end{align*}
Now, the main question is, whether the  function $t\mapsto y(t;x)$ is in $L^p([0,\alpha],\mscr{X})$ for any $\al>0$ and any $x\in \mscr{X}$. As the observation
 operator $\mscr{C}$ is unbounded, it is not clear how to extend $y(\cdot)$ to an $L^p$-function and we shall need some additional conditions on $\mscr{C}$. We thus define the following class of operators (see \cite{Weis1}).
\begin{defi}\label{C-admissible}
The observation operator $\mscr{C}\ne 0$ is called \emph{$p$-admissible} (or just admissible if there is no ambiguity) for the generator $\mscr{A}$ if for some $\t>0$ (hence for all $\t>0$) there exists a constant $\ga:=\ga(\t)>0$ such that
\begin{align*}
\|
\mscr{C}\mscr{T}(\cdot)x\|_{L^p([0,\t],\mscr{U})}\le \ga \|x\|
\end{align*}
for any $x\in D(\mscr{A})$. In this case, we call the pair $(\mscr{C},\mscr{A})$ \emph{admissible}.\end{defi}
Define the operator $$ (\Psi x)(\cdot):=y(\cdot;x),\quad x\in \mscr{X}.$$ If $(\mscr{C},\mscr{A})$ is admissible, by the density of the domain $D(\mscr{A})$, we can extend $\Psi$ to a linear bounded operator  $\Psi \in \calL(\mscr{X},L^p([0,\t],\mscr{U}))$ for any $\t>0$. Thus, the admissibility of $(\mscr{C},\mscr{A})$ is essential in any study of systems with unbounded observation operator $\mscr{C}\in\calL(D(\mscr{A}),\mscr{U})$.

The \emph{Yosida extension} of $\mscr{C}$ with respect to $\mscr{A}$ is the operator defined as
\begin{align}\label{eq:yosida}
\begin{split}
D(\mscr{C}_\Lambda)&:=\{x\in \mscr{X}: \lim_{s\to+\infty} s\,\mscr{C}R(s,\mscr{A})x\;\text{exists in}\;\mscr{U}\},\\
 \mscr{C}_\Lambda x &:=\lim_{s\to+\infty} s\,\mscr{C}R(s,\mscr{A})x.
\end{split}
\end{align}
According to Weiss \cite[(5.6) and Proposition 5.3]{Weis5}
(see also \cite[Theorem 5.4.8]{Staf}), the admissibility of $(\mscr{C},\mscr{A})$ implies that  $\mscr{T}(t)x\in D(\mscr{C}_\Lambda)$ and
\begin{align*}
\left(\Psi x\right)(t)=\mscr{C}_\Lambda \mscr{T}(t)x
\end{align*}
for all $x\in \mscr{X}$ and a.e. $t\ge 0$.

Let us now add a control term to the boundary conditions of the system \eqref{Boundary-Obser} obtaining the following input-output boundary system
\begin{align}\label{Boundary-Obser-contr}
\begin{cases} \dot{x}(t)=\mscr{A}_m x(t),\quad x(0)=x^0,& t\ge 0,\cr \mscr{G}x(t)=u(t),& t\ge 0,\cr y(t)=\mscr{M}x(t),& t\ge 0,\end{cases}
\end{align}
for some control function $u\colon [0,\infty)\to \mscr{U}$.
Formally, the system \eqref{Boundary-Obser-contr} is \emph{well-posed} if it has a unique integral solution $[0,\infty)\ni t\mapsto x(t;x^0,u)\in \mscr{X}$ which is jointly continuous at $(t,u)\in [0,+\infty)\times L^p([0,\infty),\mscr{U})$, and the output function $t\to y(t,x^0,u)$ is extended to an $L^p$-function satisfying \begin{align*} \|y(\cdot;x^0,u)\|_{L^p([0,\t],\mscr{U})}\le c \left( \|x^0\|+\|u\|_{L^p([0,\t],\mscr{U})}\right)\end{align*} for any $(x^0,u)\in \mscr{X}\times L^p([0,\t),\mscr{U})$ and any $\t>0$.

Since the control is acting at the boundary, the solution can leave the state space $X$ and
some additional conditions must be imposed to enforce the solution to remain in $X$. To this end,  we first assume that Assumptions \ref{Greiner} are satisfied and thus the following \emph{Dirichlet operator}
\begin{align}\label{Dirichlet}
\mscr{D}_\la:=\left(\mscr{G}_{|\ker(\la-\mscr{A}_m)}\right)^{-1}\in\calL(\mscr{U},\mscr{Z})
\end{align}
exists. It follows from \cite[Lemmas 1.2 and 1.3]{Gre} that
$$ \mscr{Z}=D(\mscr{A})\oplus \ker(\lambda-\mscr{A}_{m}), \quad \lambda\in\rho(\mscr{A}).$$
Since $\mscr{D}_\lambda$ takes its values in $Z$ and not in $D(A)$, and in order to rewrite system \eqref{Boundary-Obser-contr} in a consistent way, one has to work in the extrapolation space $X_{-1}$. To this purpose, let us consider the following operator
\begin{align}\label{operator-B}
	\mscr{B}:=(\lambda-\mscr{A}_{-1})\mscr{D}_\lambda\in\mathcal{L}(\mscr{U},\mscr{X}_{-1}), \qquad \lambda\in\rho(\mscr{A}).
\end{align}
For any $u\in \mscr{U}$ and  $\lambda\in\rho(\mscr{A}),$ we have $\lambda\mscr{D}_\lambda u=\mscr{A}_{m}\mscr{D}_\lambda u$, and hence,
\begin{align*}
	\left(\mscr{A}_{m}-\mscr{A}_{-1}\right)\mscr{D}_\lambda u=(\lambda-\mscr{A}_{-1})\mscr{D}_\lambda u=\mscr{B}u.
\end{align*}
From the above direct sum we deduce
\begin{align*}
	\mscr{A}_m=(\mscr{A}_{-1}+\mscr{B}\mscr{G})_{|\mscr{Z}}.
\end{align*}

Now we can rewrite the boundary system \eqref{Boundary-Obser-contr} as the following distributed linear system
\begin{align}\label{A-C-B}
\begin{cases} \dot{x}(t)=\mscr{A_{-1}} x(t)+\mscr{B}u(t),\quad x(0)=x^0,& t\ge 0,\cr y(t)=\mscr{C}x(t),& t\ge 0.\end{cases}
\end{align}
The integral solution of \eqref{A-C-B} is {given by}
\begin{align*}
x(t)=\mscr{T}(t)x^0+\int^t_0 \mscr{T}_{-1}(t-s)\mscr{B}u(s)ds\in \mscr{X}_{-1}\end{align*} for $t\ge 0$ and $x^0 \in \mscr{X}$.
\begin{defi}\label{A-B}
The control operator $\mscr{B}\in\calL(\mscr{U},\mscr{X}_{-1})$ is called \emph{$p$-admissible} (or just admissible if there is no ambiguity) for the generator $\mscr{A}$ if for some $\t>0$ we have
\begin{align*}
\Phi_\t^{A,B} u:=\int^\t_0 \mscr{T}_{-1}(\t-s)\mscr{B}u(s)ds\in \mscr{X}
\end{align*}
 for any $u\in L^p([0,+\infty),\mscr{U})$. In this case we also say that the pair $(\mscr{A},\mscr{B})$ is \emph{admissible}.
\end{defi}
Note that, by the closed graph theorem the admissibility of $(\mscr{A},\mscr{B})$ implies that $\Phi_\t^{A,B}\in \calL(L^p([0,+\infty),\mscr{U}),\mscr{X})$ for any $\t\ge 0$. Thus, if $(\mscr{A},\mscr{B})$ is admissible, the mild solution of the linear system \eqref{Boundary-Obser-contr} satisfies $x(t):=x(t,x^0,u)\in \mscr{X}$ for any $t\ge 0,x^0\in \mscr{X}$ and $u\in L^p([0,\infty),\mscr{U})$. Therefore, the admissibility of $(\mscr{A},\mscr{B})$ is essential in any study of systems with unbounded control operator $B\in\calL(\mscr{U},\mscr{X}_{-1})$.

Let us now assume that $(\mscr{A},\mscr{B})$ and $(\mscr{C},\mscr{A})$ are admissible and let us show how to extend the output function $t\mapsto y(t,x^0,u)$ to an $L^p$-function for any $x^0\in X$ and $u\in L^p_{loc}([0,\infty),\mscr{U})$. To this end, we first define the following spaces
\begin{align*} W^{1,p}_{0,\al}(\mscr{U}):=\{ u\in W^{1,p}([0,\al],\mscr{U}):u(0)=0\},\quad \al>0.\end{align*}
By using integration by part, one can see that $\Phi_t^{A,B} u\in \mscr{Z}$ for any $u\in W^{1,p}_{0,\al}(\mscr{U})$ and $t\in [0,\al]$. We then define the following {\it input-output } operator
\begin{align}\label{in-out}
(\F^{A,B,C}_{\alpha} u)(t):=\mscr{M}\Phi_t^{A,B} u,\quad t\in [0,\al],\quad u\in W^{1,p}_{0,\al}(U).
\end{align}
\begin{defi}\label{well-posed-ABC}
We say that {the triple operator} $(\mscr{A},\mscr{B},\mscr{C})$ (or the system \eqref{A-C-B}) is \emph{well-posed} on $\mscr{X}$, $\mscr{U}$,  if $(\mscr{A},\mscr{B})$ and $(\mscr{C},\mscr{A})$ are both $p$-admissible, and
for any $\al>0,$ there exists a constant $\kappa:=\kappa(\al)>0$ such that
\begin{align}\label{F-estimate}
\| \F^{A,B,C}_{\alpha} u\|_{L^p([0,\al],\mscr{U})}\le \kappa \|u\|_{L^p([0,\al],\mscr{U})}
\end{align}
for any $u\in W^{1,p}_{0,\al}(\mscr{U})$.
\end{defi}
Observe that if the triple $(\mscr{A},\mscr{B},\mscr{C})$  is well-posed then we can extend $\F^{A,B,C}_{\alpha}$ to linear bounded operator $\F^{A,B,C}_{\alpha}\in \calL(L^p([0,\al],\mscr{U}),L^p([0,\al],\mscr{U}))$ for any $\al>0$. Moreover, the output function of \eqref{A-C-B} is given by
\begin{align*}
y(t)=:y(t;x^0,u)=\mscr{C}_\Lambda \mscr{T}(t)x^0+(\F^{A,B,C}_{\alpha} u)(t)
\end{align*}
for any $x^0\in \mscr{X}, u\in L^p_{loc}([0,\infty),\mscr{U})$ and a.e. $t>0$.
To give a complete representation of the output function $t\mapsto y(t),$ we need the following concept introduced by Weiss \cite[Definition 4.2]{Weis5}.
\begin{defi}\label{regular-A-B-C}
A well-posed triple operator $(\mscr{A},\mscr{B},\mscr{C})$ is called {\em regular} (with feedthroughout zero) if
\begin{align*}
\lim_{t\to 0} \frac{1}{t}\int^t_0 \left(\F^{A,B,C}_{\alpha} (\1_{\R^+}(\cdot)z)\right)(\si)\; d\si=0
\end{align*}
for a constant input $z\in \mscr{U}$, where $\1_{\R^+}$ is the constant function equal to $1$ on $\R_+$.
\end{defi}

The following theorem gives a characterization of the regularity for a well-posed triple, see \cite[Theorem 5.6.5]{Staf}.
\begin{theo}\label{regular}
A well-posed triple $(A,B,C)$ is regular if and only if
$\mscr{D}_\lambda v=R(\la,\mscr{A}_{-1})\mscr{B} v\in D(\mscr{C}_\Lambda)$ for all $v\in \mscr{U}$ and some $\la\in\rho(\mscr{A})$.
\end{theo}

According to Weiss \cite[Theorem 5.5]{Weis5},
if the triple $(\mscr{A},\mscr{B},\mscr{C})$ is regular then $\Phi_t^{A,B} u\in D(\mscr{C}_\Lambda)$ and $(\F^{A,B,C}_{\alpha} u)(t)=\mscr{C}_\Lambda \Phi_t^{A,B} u$ for all $u\in L^p_{loc}(\R^+,\mscr{U})$ and  a.e. $t\ge 0$. This implies that the state function $x(\cdot)$ and the output function $y(\cdot)$ of the system \eqref{A-C-B} satisfy  $x(t)\in D(\mscr{C}_\Lambda)$ and
\begin{align*}
y(t)=\mscr{C}_\Lambda x(t)
\end{align*}
for all $x^0\in \mscr{X}, u\in L^p_{loc}(\R^+,\mscr{U})$ and  a.e. $t\ge 0$.

We shall further assume that $I_U-\F^{A,B,C}_{\t}$ is invertible in $L^p([0,\t],\mscr{U})$ for some (and hence all) $\t>0$. Then the operator
\begin{align}\label{closed-oper}
\tilde{\mscr{A}}=\mscr{A}_{-1}+\mscr{B}\mscr{C}_\Lambda,\quad D(\tilde{\mscr{A}})=\{x\in D(\mscr{C}_\Lambda): \tilde{\mscr{A}} x\in \mscr{X}\}
\end{align}
generates a strongly continuous semigroup $(\tilde{\mscr{T}}(t))_{t\ge 0}$ on $\mscr{X}$ such that $\tilde{\mscr{T}}(s)x\in D(\mscr{C}_\Lambda)$ for all $x\in \mscr{X}$ and a.e. $s\ge 0$.  Moreover, for $\t>0$ there exists $c_\t>0$ such that
\begin{align*}
\|\mscr{C}_\Lambda \tilde{\mscr{T}}(\cdot)x\|_{L^p([0,\t],\mscr{U})}\le c_\t \|x\|
\end{align*}
for all $x\in \mscr{X}$. In addition we have
\begin{align*}
\tilde{\mscr{T}}(t)x=\mscr{T}(t)x+\int^t_0 \mscr{T}_{-1}(t-s)\mscr{B}\mscr{C}_\Lambda \tilde{\mscr{T}}(s)xds
\end{align*}
for all $x\in \mscr{X}$ and $t\ge 0$. For more details and references see e.g. \cite{Weis5} (for Hilbert setting) and \cite[Chap. 7]{Staf} (for Banach setting).

We introduce the augmented assumption on operators $A,B,C$ as follows.
\begin{asp}\label{tildeA}
In addition to Assumptions \ref{Greiner}, the triple $(\mscr{A},\mscr{B},\mscr{C})$ is regular on $\mscr{X}$, $\mscr{U}$,
and operator $I_U-\F^{A,B,C}_{\t}$ is invertible in $L^p([0,\t],\mscr{U})$ for some $\t>0$.
\end{asp}
Finally, we consider the  boundary perturbed problem
\begin{align}\label{Perturbed-boundary-abst}
\begin{cases}
\dot{x}(t)=\mscr{A}_m x(t),\quad x(0)=x^0,& t>0,\cr \mscr{G}x(t)=\mscr{M}x(t),& t\ge 0.\end{cases}
\end{align}
To this equation we associate the following linear operator
\begin{align}\label{AA-M}
\mscr{A}^M\subseteq\mscr{A}_m,\quad D(A^M)=\{x\in \mscr{Z}: \mscr{G}x=\mscr{M}x\}.
\end{align}
Problem \eqref{Perturbed-boundary-abst} has a unique mild solution if and only if operator $A^M$ generates a strongly continuous semigroup on $\mscr{X}$. The following theorem 
provides sufficient conditions for the generation property and reveals the spectrum  of  $A^M$.
\begin{theo}\label{HMR15}
Let Assumptions \ref{tildeA} be satisfied.
Then the operator $A^M$ defined in \eqref{AA-M} generates a  strongly continuous semigroup $(T^M(t))_{\ge 0}$, given by
\begin{align*}
\mscr{T}^M(t)x=\mscr{T}(t)x+\int^t_0 \mscr{T}_{-1}(t-s)\mscr{B}\mscr{C}_\Lambda \mscr{T}^M(s)x\,ds
\end{align*}
for all $x\in \mscr{X}$ and $t\ge 0$. Moreover, for any $\lambda\in \rho(A)$,
$$\lambda\in \rho(A^M) \Longleftrightarrow 1\in \rho(MD_\lambda) \Longleftrightarrow 1\in \rho(D_\lambda M).$$
In this case, the resolvent operator of $A^M$ 
 is given by $$R(\lambda,A^M)=(I-D_\lambda M)^{-1}R(\lambda,A).$$
 Additionally, the operator $B$ is an admissible control operator for $A^M$ and if we choose $D_\lambda^M$ as $R(\lambda,A_{-1}^M)B$, then

 $$D^M_\lambda=D_\la (I-MD_\lambda)^{-1}=(I-D_\la M)^{-1}D_\la$$ for any $\lambda\in \rho(A)\cap \rho(A^M)$.

\end{theo}
\begin{proof}
First note that the operator $\mscr{A}^{\mscr{M}}$ coincides with the operator $\tilde{\mscr{A}}$ defined in \eqref{closed-oper}, hence
the generating properties and spectral theory associated with the operator $A^M$ follow by \cite[Theorem 4.1]{HMR}.
On the other hand, the admissibility of the control operator $B$ for $A^M$ can be inferred from \cite[Theorem 6.1]{Weis5}, see also \cite[Theorem 7.1.2]{Staf}. According to this reference, the control maps  associated with $B$ and $A^M$ are expressed as $$\Phi^{A,B}_t=\Phi^{A^M,B}_t (I-F_t^{A,B,C}).$$
By taking the Laplace transform of both sides of the equation, we obtain for $\la\in \rho(A)\cap\rho(A^M)$, $$ D_\la= D^M_\la-D^M_\la MD_\la,$$ 
see also \cite[Theorem 3.2]{BBDHL}. This concludes the proof.
\end{proof}

\section{Abstract boundary delay equations}\label{sec:3}

This section aims to introduce a perturbation approach that establishes the well-posedness of the abstract boundary delay equation \eqref{BED}.
Since the left translation semigroup plays a seminal role in understanding delayed processes, we start by presenting a classical illustration of a regular linear system governed by it.
\begin{exa}\label{shift-S}
Let $z\colon [-1,+\infty)\to X$ be a given function. We define another function $z_t\colon [-1,0]\to X$ as follows:
 \begin{align*} z_t(\theta):=\begin{cases} z(t+\theta),& -t\le \theta\le 0,\cr \varphi(t+\theta),& -1\le \theta< -t,\end{cases}  \end{align*} where  $0\le t< 1$ and $\varphi\in L^p([-1,0],X)$ for a Banach space $X$ and a real number $p\ge 1$. The function $\varrho(t,\theta)=z_t(\theta)=z(t+\theta)$ for $t\ge 0$ and $\theta\in [-1,0]$ serves as the solution to the  boundary value problem
\begin{align}\label{translation-equation}
\begin{cases}
\frac{\partial}{\partial t} \varrho(t,\theta) = \frac{\partial}{\partial \theta} \varrho(t,\theta),&t \geq 0,\; \theta \in [-1,0],\\
 \varrho(t,0) = z(t),& t \geq 0,\\
\varrho(0,\theta) = \varphi(\theta),&\theta \in [-1,0].
\end{cases}
\end{align}
It is worth noting that the {solution} can be rewritten as
\begin{align}\label{shift-state} z_t=S(t)\varphi+\phi_t z(\cdot),\qquad t\ge 0,\end{align}
where $(S(t))_{t\ge 0}$ represents the left {translation} semigroup on $L^p([-1,0],X)$ and is defined as
\begin{align}\label{shift-sg} (S(t)\varphi)(\theta):=\begin{cases} 0,& -t\le \theta\le 0,\cr \varphi(t+\theta),& -1\le \theta< -t,\end{cases} \end{align}
Additionally, $\phi_t \colon L^p([0,+\infty),X)\to L^p([-1,0],X)$ is defined as
\begin{align}\label{control-maps-shift} (\phi_t z(\cdot))(\theta):=\begin{cases} z(t+\theta),& -t\le \theta\le 0,\cr 0,& -1\le \theta< -t.\end{cases}\end{align}
Let us now employ Greiner's approach, as demonstrated above, to reformulate the boundary system \eqref{translation-equation} as a distributed system. The trace operator, denoted by the Dirac operator $\delta_0 f=f(0)$, is surjective from the Sobolev space $W^{1,p}([-1,0],X)$ to $X$. The operator $$ Q\varphi=\varphi',\quad D(Q)=\ker\delta_0=\{\varphi\in W^{1,p}([-1,0],X):\varphi(0)=0\}$$ generates the left translation semigroup $(S(t))_{t\ge 0}$ on $L^p([-1,0],X)$. The Dirichlet operator associated with $ \frac{\partial}{\partial \theta}$ and $\delta_0$ is denoted as $e_\la\colon  X\to L^p([-1,0],X)$, where $(e_\la x)(\theta):=e^{\lambda \theta}x$ for all $\theta\in [-1,0]$. By selecting $\beta:=(\la-Q_{-1})e_\la$ for $\la\in \rho(Q)=\mathbb{C}$, the control maps $\phi_t$ defined in \eqref{control-maps-shift} can also be expressed as $$ \phi_t z=\int^t_0 S_{-1}(t-s)\beta z(s)ds,\qquad t\ge 0.$$ Therefore, $\beta$ serves as an admissible control operator for $Q$.
Note, that one can show by the resolvent equation that $\beta$ does not depend on $\lambda$, see \cite[Lem.~1.3]{Gre}.
\end{exa}

\begin{rem}
To be consistent with the notations of the previous section, the operators $\delta_0,\,Q,\,\beta ,\,e_\la$, and $\phi_t$ are respectively equal to the operators $G,\,A,\,B,\,D_\la$, and $\Phi_t^{A,B}$ used in the previous section. The motivation for using different notations will become clear in the sequel.
\end{rem}

To bring the delay operators into the picture, we move to larger Banach spaces and define
\begin{align*}
& \mathcal{Z}:=Z\times W^{1,p}([-1,0],X),\cr
& \mathbb{X}:=X\times{Y},\qquad \|(\begin{smallmatrix}x\\f\end{smallmatrix})\|=\|x\|+\|f\|_p,
\end{align*}
where $p\ge 1$ and 
\begin{align*}
Y:=L^p([-1,0],X)\quad\text{is endowed with the norm}\quad \|f\|_p^p= \int^0_{-1}\|f(\si)\|^p\; d\si.
\end{align*}
\begin{theo}\label{Main-result-1}
Let the Assumption \ref{tildeA} be satisfied {and let $P$ and $L$ be the operators defined in \eqref{eq:opPL}}. Then the operator $\calA_{P,L}\colon D(\calA_{P,L})\subset \mathbb{X}\to\mathbb{X}$, defined by
\begin{align}\label{Big-generator}
\begin{split}
\calA_{P,L}&:=\begin{pmatrix} A_m&P\\ 0& \frac{d}{ds}\end{pmatrix},\cr D(\calA_{P,L})&:=\left\{(\begin{smallmatrix}x\\\varphi\end{smallmatrix})\in \mathcal{Z}:(G-M)x=L\varphi,\;x=\varphi(0)\right\}
\end{split}\end{align}
generates a strongly continuous semigroup $(\mathcal{T}_{P,L}(t))_{t\ge 0}$ on $\mathbb{X}$ and for any initial condition $(\begin{smallmatrix}x\\\varphi\end{smallmatrix})\in \mathbb{X}$, we have
\begin{align}\label{Sol}
(\begin{smallmatrix}z(t)\\z_t\end{smallmatrix})=\mathcal{T}_{P,L}(t)(\begin{smallmatrix}x\\\varphi\end{smallmatrix}),\quad t\ge 0,
\end{align}
where $z(\cdot)$ is the mild solution to \eqref{BED}.
\end{theo}
\begin{proof}
We split the proof in two main steps.

\noindent\underline{Step 1.} We reconfigure the original boundary value problem \eqref{BED} into the following input-output system.\begin{align} \label{BED1}
\begin{cases}
\dot{z}(t) = A_m z(t) + u_2(t),\quad z(0) = x, & t \geq 0, \cr
G z(t) = M z(t) + u_1(t), & t \geq 0, \cr
y(t) = \begin{pmatrix} Lz_t \\ Pz_t \end{pmatrix},& t\ge 0,\cr z(s) = \varphi(s) & s \in [-1,0],
\end{cases}
\end{align}
subject to the input $u(t) = \begin{pmatrix} u_1(t) \\ u_2(t) \end{pmatrix} \in U \times X$, which serves as the output $y(t)$. By considering the matrix $$B^I:=\begin{pmatrix} B & I \end{pmatrix}\colon U\times X\to X_{-1},$$
with $B$ defined in \eqref{operator-B},
we can use the construction in Section 2, Theorem \ref{HMR15}
and \cite[the proof of Theorem 4.3]{HMR} to transform the system \eqref{BED1} into the following form
\begin{align} \label{BED2}
\begin{cases}
\dot{z}(t) = A_{-1}^M z(t) + B^I u(t), \quad z(0) = x, & t \geq 0, \cr
u(t) = y(t) = \begin{pmatrix} Lz_t \\ Pz_t \end{pmatrix}, & t \geq 0, \cr
z(s) = \varphi(s) & s \in [-1,0],
\end{cases}
\end{align}
where the operator $A^M$ is defined in \eqref{AA-M}. The mild solution of \eqref{BED2} is given by \begin{align}\label{z-sol-A-M}\begin{split} z(t)&=T^M(t)x+\int^t_0 T^M_{-1}(t-s) B^I u(s)ds,\cr &=:T^M(t)x+\Phi^{A^M,B^I}_tu\end{split}\end{align} for any $t\ge 0$, see \cite[Theorem 4.3]{HMR}.

We now introduce the operators
\begin{align*}
& \mathcal{B}:=\begin{pmatrix} B& I\\ 0&0\end{pmatrix}\colon  U\times X\to X_{-1}\times{Y},\cr &  \mathbb{P}:=\begin{pmatrix} 0& L\\ 0&P\end{pmatrix}\colon \mathcal{Z}\to U\times X.
\end{align*}
If we set \begin{align*}
W\colon [0,\infty)\ni t\mapsto W(t)= (\begin{smallmatrix}z(t)\\z_t\end{smallmatrix}),
\end{align*}
then the system \eqref{BED2} becomes
\begin{align}\label{BED3}
\begin{cases}
\dot{W}(t)=\mathcal{A} W(t)+\mathcal{B}u(t), &t\ge 0,\,\\
 W(0)=\left(\begin{smallmatrix} x\\ \varphi\end{smallmatrix}\right),&\\
u(t)=y(t)=\mathcal{C} W(t),& t\ge 0,
\end{cases}
\end{align}
where
\begin{align*} \mathcal{A}&:=\begin{pmatrix} A^M&0\\ 0& \frac{d}{d\theta}\end{pmatrix},\\
D(\mathcal{A})&:=\left\{\left(\begin{smallmatrix} x\\ \varphi\end{smallmatrix}\right)\in D(A^M)\times W^{1,p}([-1,0],X):\varphi(0)=x \right\},\end{align*}
and
\begin{align*}
\mathcal{C}:=\mathbb{P}_{|D(\mathcal{A})}.
\end{align*}
Due to Assumption \ref{tildeA}, it is well-known, see \cite[Theorem 3.25]{BP}, that the operator  $\mathcal{A}$ generates a strongly continuous semigroup $(\mathcal{T}(t))_{t\ge 0}$ on $\mathbb{X}$ given by
\begin{align*}\mathcal{T}(t)=\begin{pmatrix}T^M(t)& 0\\ T_t^M & S(t) \end{pmatrix},\quad t\ge 0, \end{align*} where $T^M(\cdot)$ is the strongly continuous semigroup generated by $A^M$ (see  Theorem \ref{HMR15}), $T^M_t\colon  X\to L^p([-1,0],X)$ is the family of operators defined as $$ (T^M_t x)(\theta)=\begin{cases} T^M(t+\theta)x,& -t\le \theta\le 0,\cr 0,& -1\le \theta< -t \end{cases}$$
for any $t\ge 0$, {and $(S(t))_{t\ge 0}$ is the translation semigroup given in Example \ref{shift-S}. Then  the solution to \eqref{BED3} is obtained as}
\begin{align}\label{W-sol}\begin{split} W(t)&=\mathcal{T}(t)(\begin{smallmatrix} x\\ \varphi\end{smallmatrix})+\int^t_0 \calT_{-1}(t-s)\calB u(s)ds\cr &=: \mathcal{T}(t)(\begin{smallmatrix} x\\ \varphi\end{smallmatrix})+\Phi^{\calA,\calB}_t u \end{split}\end{align}
for any $t\ge 0$ and $(\begin{smallmatrix} x\\ \varphi\end{smallmatrix})\in \mathbb{X}$. On the other hand, from \eqref{z-sol-A-M}, \eqref{shift-state}, and \eqref{control-maps-shift} we obtain
\begin{align*}
W(t)&= \begin{pmatrix}z(t)\\z_t\end{pmatrix} = \begin{pmatrix}T^M(t)x+\Phi^{A^M,B^I}_t u\\ S(t)\varphi+\phi_t z(\cdot) \end{pmatrix}\cr
&=\begin{pmatrix}T^M(t)x \\ T^M_t x+S(t)\varphi \end{pmatrix}+\begin{pmatrix}\Phi^{A^M,B^I}_t u\\ \phi_t \Phi^{A^M,B^I}_\cdot u \end{pmatrix}\cr &=\mathcal{T}(t)\left(\begin{smallmatrix} x\\ \varphi\end{smallmatrix}\right)+\begin{pmatrix}\Phi^{A^M,B^I}_t u\\ \phi_t \Phi^{A^M,B^I}_\cdot u \end{pmatrix}.
\end{align*}
Combining this with \eqref{W-sol}, we have
\begin{align}\label{mathscrB}
\Phi^{\mathcal{A},\mathcal{B}}_t u=\begin{pmatrix}\Phi^{A^M,B^I}_t u\\ \phi_t \Phi^{A^M,B^I}_\cdot u \end{pmatrix}.
\end{align}

\noindent\underline{Step 2.} Let us now verify the conditions of  Theorem \ref{HMR15}.
From \eqref{mathscrB} and the fact that $(A^M,B^I)$ and $(Q,\beta)$ are admissible, we deduce that $\Phi^{\mathcal{A},\mathcal{B}}_t u\in \mathbb{X}$ for any $t\ge 0$ and $u\in L^p(\R^+,U\times X)$. Thus $(\mathcal{A},\mathcal{B})$ is admissible. Furthermore, {by the same computation as} in \cite[Lemma 6.2]{Had}, one can see that $(\mathcal{C},\mathcal{A})$ is admissible. Now, let us construct an input-output operator for the triple $(\mathcal{A},\mathcal{B},\mathcal{C})$. To this purpose, we need to compute the Yosida extension of $\mathcal{C} $ with respect to the generator $\mathcal{A}$. For sufficiently large  real $\la>0,$  $x\in X$ and $\varphi\in{Y}$, we have by \cite[Proposition 3.19]{BP},
\begin{align*}
R(\la,\calA)=\begin{pmatrix} R(\la,A^M)& 0\\ e_\la R(\la,A^M)& R(\la,Q)\end{pmatrix}
\end{align*}
and \begin{align*}
\mathcal{C}\la R(\la,\calA)\begin{pmatrix}x\\ \varphi\end{pmatrix}=\begin{pmatrix}L e_\la (\la R(\la,A^M)x)+L\la R(\la,Q)\varphi\\ P e_\la (\la R(\la,A^M)x)+P\la R(\la,Q)\varphi\end{pmatrix},
\end{align*}
{where $(e_\la x)(s):=e^{\lambda s}x$ for $s\in [-1,0]$}.
Since $\|Le_\la\|_{\mathcal{L}(X,U)}\to 0$ and $\|Pe_\la\|_{\mathcal{L}(X)}\to 0$ as $\la\to+\infty$ (see \cite[Lemma 6.1]{Had}), it follows
\begin{align*}
\lim_{\la\to+\infty} \|L e_\la (\la R(\la,A^M)x)\|_U=0=\lim_{\la\to+\infty} \|P e_\la (\la R(\la,A^M)x)\|_X,\quad x\in X.
\end{align*}
Thus, \begin{align*}
\begin{pmatrix}x\\ \varphi\end{pmatrix}\in D(\mathcal{C}_\Lambda) \Longleftrightarrow x\in X\quad\text{and}\quad \varphi\in D(L_\Lambda)\cap D(P_\Lambda),
\end{align*}
where $L_\Lambda$ and $P_\Lambda$ are the Yosida extensions of $L$ and $P$ for $Q$, respectively. Therefore,
\begin{align}\label{mathscrCLambda}
 D(\mathcal{C}_\Lambda)=X\times D(L_\Lambda)\cap D(P_\Lambda)\quad\text{and}\quad \mathcal{C}_\Lambda=\begin{pmatrix} 0& L_\Lambda\\ 0& P_\Lambda\end{pmatrix}.
\end{align}
On the other hand,  both triples $(Q,\beta,L)$ and $(Q,\beta,P)$ are regular, cf.~\cite[Theorem 4]{HIR}, hence,  $\phi_t v\in D(L_\Lambda)\cap D(P_\Lambda)$ for any $v\in L^p_{loc}(\R^+,X)$ and a.e. $t>0$. By \eqref{mathscrB} and \eqref{mathscrCLambda}, we have $\Phi^{\mathcal{A},\mathcal{B}}_t u\in D(\mathcal{C}_\Lambda)$ for a.e. $t>0$ and all $u\in L^p_{loc}(\R^+,U\times X)$. {We can thus define}
\begin{align*}
\mathbb{F}_t^{\mathcal{A},\mathcal{B},\mathcal{C}} u&:=\mathcal{C}_\Lambda \Phi^{\mathcal{A},\mathcal{B}}_t u\cr &= \begin{pmatrix}\F_t^{L}\Phi^{A^M,B^I}_\cdot u\\ \F_t^{P}\Phi^{A^M,B^I}_\cdot u \end{pmatrix}
\end{align*}
for a.e. $t>0$ and $u\in L^p([0,t],U\times X)$, where $\F_t^L:=\F_t^{Q,\beta ,L}$ and $\F_t^P:=\F_t^{Q,\beta ,P}$ are the extended input-output operators associated with the regular triples $(Q,\beta,L)$ and $(Q,\beta,P)$, respectively.  Clearly, for any $\al>0$, we have
\begin{align*} \mathbb{F}_{\al}^{\mathcal{A},\mathcal{B},\mathcal{C}}\in \mathcal{L}(L^p([0,\al],U\times X)).\end{align*} This shows that the triple $(\mathcal{A},\mathcal{B},\mathcal{C})$ is well-posed. By {taking Laplace transform on both sides of} \eqref{mathscrB}, we have
\begin{align}\label{DD_la}\begin{split} 
\mathbb{D}_\la \begin{pmatrix} v_1\\ v_2\end{pmatrix} &:=R(\la,\mathcal{A}_{-1})\mathcal{B}\begin{pmatrix} v_1\\ v_2\end{pmatrix}\\
&=\begin{pmatrix} R(\la,A^M_{-1})Bv_1+R(\la,A^M_{-1})v_2\\ e_\la (R(\la,A^M_{-1})Bv_1+R(\la,A^M_{-1})v_2)\end{pmatrix}\in D(\mathcal{C}_\Lambda)
\end{split}
\end{align}
for any $v_1\in U$ and $v_2\in X$. Thus by Theorem \ref{regular},
the triple $(\mathcal{A},\mathcal{B},\mathcal{C})$ is regular.
Now observe, that
\begin{align*}
\mathbb{F}_{\al}^{\mathcal{A},\mathcal{B},\mathcal{C}}=\begin{pmatrix}\F_{\al}^L\Phi^{A^M,B}_\cdot & \F_{\al}^L\Phi^{A^M,I}_\cdot\\ \F_{\al}^P\Phi^{A^M,B}_\cdot & \F_{\al}^P\Phi^{A^M,I}_\cdot\end{pmatrix}.
\end{align*}

By \eqref{in-out}, and by applying H\"older's and Fubini's theorem, we have
\begin{eqnarray*}
\|F_\alpha ^Pu\|_{L^p([0,\alpha],X)}^p &=& \|P\phi_{\cdot} u\|_{L^p([0,\alpha],X)}^p \\
&=& \int_0^\alpha \left\|\int_{-1}^0 d\mu(\theta)(\phi_s u)(\theta)\right\|^pds\\
&\le & \int_0^\alpha |\mu|([-s ,0])^{\frac{p}{q}}\int_{-s}^0 d|\mu|(\theta)\|u(s+\theta)\|^p ds\\
&\le & |\mu|([-\alpha ,0])^{\frac{p}{q}}\int_{-\alpha}^0 d|\mu|(\theta)\left(\int_{-\theta}^\alpha\|u(s+\theta)\|^pds\right)\\
&\le & |\mu|([-\alpha ,0])^{\frac{p}{q}+1}\|u\|_{L^p([0,\alpha],X)}^p
\end{eqnarray*}
for any $\alpha\in (0,1)$ and any $u\in W^{1,p}_{0,\alpha}(X)$, where $\frac{1}{p}+\frac{1}{q}=1$, $1\le p,q \le \infty$.
Hence, by density, we obtain
$$\|F_\alpha ^Pu\|_{L^p([0,\alpha],X)}\le  |\mu|([-\alpha ,0])\|u\|_{L^p([0,\alpha],X)},\quad  \text{ for all } u\in L^p([0,\alpha],X).$$
By the same computation one gets
$$\|F_\alpha ^Lu\|_{L^p([0,\alpha],U)}\le |\nu|([-\alpha ,0])\|u\|_{L^p([0,\alpha],U)},\quad  \text{ for all } u\in L^p([0,\alpha],U).$$
Thus, there exists a constant $C>0$ such that for any $\al\in (0,1)$ and any $u_1\in L^p([0,\al],U)$, $u_2\in L^p([0,\al],X)$,
we have
\begin{align*}
&\left\| \mathbb{F}_{\al}^{\mathcal{A},\mathcal{B},\mathcal{C}}\left(\begin{smallmatrix} u_1\\ u_2\end{smallmatrix}\right)\right\|_{L^p([0,\al],U\times X)} \\
&\le C \left(|\mu|([-\al,0])+|\nu|([-\al,0])\right) \left\|\left(\begin{smallmatrix} u_1\\ u_2\end{smallmatrix}\right)\right\|_{L^p([0,\al],U\times X)}.
\end{align*}

Since $|\mu|([-\al,0])+|\nu|([-\al,0])\to 0$ as $\al\to 0$,  we can choose $\al>0$ such that
\begin{align*}
\left\| \mathbb{F}_{\al}^{\mathcal{A},\mathcal{B},\mathcal{C}}\right\|_{\calL(L^p([0,\al],U\times X))} <\frac{1}{2}.
\end{align*}
Thus, $I_{U\times X}-\mathbb{F}_{\al}^{\mathcal{A},\mathcal{B},\mathcal{C}}\colon L^p([0,\al],U\times X))\to L^p([0,\al],U\times X))$ admits an uniformly bounded inverse. 
By \cite[Theorem 2.7]{HMR}, the operator
\begin{equation*}
 \mathcal{A}^{cl}:=\mathcal{A}_{-1}+\mathcal{B}\mathcal{C}_\Lambda
\end{equation*}
with domain
\begin{equation*}
 D(\mathcal{A}^{cl}):=\left\{(\begin{smallmatrix} x\\ \varphi\end{smallmatrix})\in D(\mathcal{C}_\Lambda):(\mathcal{A}_{-1}+\mathcal{B}\mathcal{C}_\Lambda)(\begin{smallmatrix} x\\ \varphi\end{smallmatrix})\in \mathbb{X}\right\}
\end{equation*}
generates a strongly continuous semigroup on $\mathbb{X}$. For $\la>0$ be sufficiently large, $(\begin{smallmatrix} x\\ \varphi\end{smallmatrix})\in D(\mathcal{A}^{cl})$ if and only if
\begin{align}\label{ironmaidon}
\la \mathbb{D}_\la \mathcal{C}_\Lambda(\begin{smallmatrix} x\\ \varphi\end{smallmatrix})+\mathcal{A}_{-1}\left((\begin{smallmatrix} x\\ \varphi\end{smallmatrix})-\mathbb{D}_\la \mathcal{C}_\Lambda(\begin{smallmatrix} x\\ \varphi\end{smallmatrix}) \right)\in\mathbb{X}.
\end{align}
In particular, {putting $D_\la^M:=R(\la,A^M_{-1})B$ we obtain}
\begin{align}\label{in the domain}
(\begin{smallmatrix} x\\ \varphi\end{smallmatrix})-\mathbb{D}_\la \mathcal{C}_\Lambda(\begin{smallmatrix} x\\ \varphi\end{smallmatrix})=\begin{pmatrix}x- D^M_\la L_\Lambda \varphi-R(\la,{A^M})P_\Lambda \varphi\\ \varphi-e_\la [D^M_\la L_\Lambda \varphi+R(\la,{A^M})P_\Lambda \varphi]\end{pmatrix} \in D(\mathcal{A}).
\end{align}
This implies that $x- D^M_\la L_\Lambda \varphi-R(\la,A^M)P_\Lambda \varphi\in D(A^M),$ $x\in Z$, $\varphi\in W^{1,p}([-1,0],X)$, and $\varphi(0)=x$. From this and \cite[Lemma 3.6]{HMR}, we deduce that $L_\Lambda \varphi=L\varphi,$  $P_\Lambda \varphi=P\varphi$,  $x- D^M_\la L \varphi\in D(A^M)$, so that $Gx-GD^M_\la L \varphi=Mx-MD^M_\la L\varphi$.
On the other hand, from Theorem \ref{HMR15}, we know that $D^M_\la =D_\la+D_\la MD^M_\la$. Thus, $GD^M_\la L=L\varphi+MD^{M}_\la L\varphi$ and therefore $ (G-M)x=L\varphi.$ This means that $(\begin{smallmatrix} x\\ \varphi\end{smallmatrix})\in D(\calA_{P,L})$.

Let us now prove that the operators $\mathcal{A}^{cl}$ and $\calA_{P,L}$ coincide on
$D(\mathcal{A}^{cl})$. In fact, for $(\begin{smallmatrix} x\\ \varphi\end{smallmatrix})\in D(\calA^{cl})$ and by using \eqref{ironmaidon}, we have
\begin{align*}
& \mathcal{A}^{cl}(\begin{smallmatrix} x\\ \varphi\end{smallmatrix})=\la \mathbb{D}_\la \mathcal{C}_\Lambda(\begin{smallmatrix} x\\ \varphi\end{smallmatrix})+\mathcal{A}\left((\begin{smallmatrix} x\\ \varphi\end{smallmatrix})-\mathbb{D}_\la \mathcal{C}_\Lambda(\begin{smallmatrix} x\\ \varphi\end{smallmatrix}) \right)\cr &=
\begin{pmatrix}\la [D^M_\la L \varphi+R(\la,A^M)P \varphi]\\ \la e_\la [D^M_\la L \varphi+R(\la,A^M)P \varphi] \end{pmatrix}
+\begin{pmatrix}A_M(x-[D^M_\la L \varphi+R(\la,A^M)P \varphi])\\ \frac{d}{d\theta} (\varphi- e_\la [D^M_\la L \varphi+R(\la,A^M)P \varphi] ) \end{pmatrix}\cr &=
\begin{pmatrix}\la [D^M_\la L \varphi+R(\la,A^M)P \varphi]\\ \la e_\la [D^M_\la L \varphi+R(\la,A^M)P \varphi] \end{pmatrix}
+\begin{pmatrix}A_m x-A_m D^M_\la L \varphi-A^M R(\la,A^M)P \varphi\\ \frac{d\varphi}{d\theta} - \la e_\la [D^M_\la L \varphi+R(\la,A^M)P \varphi] \end{pmatrix}\cr &= \begin{pmatrix}A_m x-(\la-A_m) D^M_\la L \varphi-(\la-A^M )R(\la,A^M)P \varphi\\ \frac{d\varphi}{d\theta}  \end{pmatrix}\cr & = \begin{pmatrix}A_m x+P\varphi\\ \frac{d\varphi}{d\theta}\end{pmatrix}\cr & = \calA_{P,L}(\begin{smallmatrix} x\\ \varphi\end{smallmatrix}),
\end{align*}
due to $(\la-A_m) D^M_\la L \varphi=0$ (in fact, ${\rm range}(D^M_\la)={\rm range}(D_\la)\subset \ker(\la-A_m)$).
By the same computations as before, one can see easily that $D(\mathcal{A}_{P,L})\subseteq D(\calA^{cl})$.
This ends the proof.
\end{proof}

\section{Spectral theory and positivity}\label{sec:4}

After providing the conditions that ensure the existence of the semigroup associated with the boundary value problem with delay \eqref{BED}, we will now seek additional conditions to secure the positivity of this semigroup. This requires an explicit calculation of the resolvent operator of its generator. The following result shows some spectral properties for the generator $\calA_{P,L}$.

{For $\la\in \rho(A)$ we denote
\begin{align}\label{eq:delta}
\begin{split}
\Delta(\la)&:=e_\la [D_\la ^M L+ R(\la,A^M)P]\\
&= e_\la (I-D_\la M)^{-1} [D_\la L+ R(\la,A)P].
\end{split}
\end{align}
The last equality follows by  Theorem \ref{HMR15}, since $D_\la ^M=(I-D_\la M)^{-1}D_\la$ and $R(\la ,A^M)=(I-D_\la M)^{-1}R(\la ,A)$.}
\begin{theo}\label{Main-result-2}
Let Assumption \ref{tildeA} be satisfied. For $\la\in \rho(A)$ and $1\in \rho(D_\la M)$, we have
\begin{align*}
\la\in \rho(\calA_{P,L}) & \Longleftrightarrow 1\in {\rho\left(\Delta(\la)\right).}
\end{align*}
In this case,
\begin{align*}
R(\la,\calA_{P,L})=\begin{pmatrix} \mscr{R}(\la)_{11}&\mscr{R}(\la)_{12}\\ \mscr{R}(\la)_{21}&\mscr{R}(\la)_{22}\end{pmatrix}
\end{align*}
with
\begin{align*}
\mscr{R}(\la)_{11}&:= {(I-\Delta(\la))^{-1} R(\la ,A^M)}\\
\mscr{R}(\la)_{12}&:= {e_{-\la}\Delta(\la)(I-\Delta(\la))^{-1} R(\la ,Q)}\\
\mscr{R}(\la)_{21}&:=(I-\Delta(\la))^{-1}e_\la R(\la,A^M),\cr
\mscr{R}(\la)_{22}&:=(I-\Delta(\la))^{-1} R(\la,Q).
\end{align*}
\end{theo}
\begin{proof}
Let $\la\in \rho(A)$ and $1\in \rho(D_\la M)$. According to  Theorem \ref{HMR15}, we have $\la\in \rho(A^M),$  $R(\la,A^M)=(I-D_\la M)^{-1} R(\la,A)$ and $D^M_\la:=R(\la,A^M_{-1})B =(I-D_\la M)^{-1} D_\la$.

Since $\mathbb{D}_\la:=R(\la ,\mathcal{A}_{-1})\mathcal{B}$ and
$\mathcal{A}_{P,L}=\mathcal{A}^{cl}$, it follows from \eqref{in the domain} that, on the domain of $\mathcal{A}_{P,L}$, we have
\begin{eqnarray*}
\la -\mathcal{A}_{P,L} &=& \la -\mathcal{A}^{cl}\\
&=& \la -\mathcal{A}_{-1}-(\la -\mathcal{A}_{-1})\mathbb{D}_\la C_\Lambda \\
&=& (\la -\mathcal{A}_{-1})(I_{U\times X}-\mathbb{D}_\la \mathbb{P})\\
&=& (\la -\mathcal{A})(I_{U\times X}-\mathbb{D}_\la \mathbb{P}).
\end{eqnarray*}
This shows that
\begin{align*}
\la\in \rho(\calA_{P,L})\Longleftrightarrow \la\in \rho(\mathcal{A}^{cl})\Longleftrightarrow 1\in \rho(\mathbb{D}_\la \mathbb{P}).
\end{align*}
Using \eqref{DD_la} we now compute
\begin{align*}
\mathbb{D}_\la \mathbb{P}&=\begin{pmatrix} D^M_\la & R(\la,A^M)\\ e_\la D^M_\la & e_\la R(\la,A^M)\end{pmatrix}\begin{pmatrix} 0&L\\ 0&P\end{pmatrix}\cr & = \begin{pmatrix} 0& D^M_\la L+R(\la ,A^M)P\\ 0& e_\la D^M_\la L+e_\la R(\la,A^M)P\end{pmatrix}
{=  \begin{pmatrix} 0& e_{-\la}\Delta(\la)\\0 &  \Delta(\la)\end{pmatrix} }
\end{align*}
Again by \cite[Theorem 4.1]{HMR}, we have
\begin{align*}
R(\la,\calA_{P,L}) &=(I_{U\times X}-\mathbb{D}_\la \mathbb{P})^{-1} R(\la,\mathcal{A})\cr
& =\begin{pmatrix} I&  {e_{-\la}\Delta(\la)}
 (I-\Delta(\la))^{-1}\\ 0& (I-\Delta(\la))^{-1}\end{pmatrix}\begin{pmatrix} R(\la,A^M)& 0\\ e_\la R(\la,A^M)& R(\la,Q)\end{pmatrix}
\end{align*}
The result now follows by a simple product matrix computation.
\end{proof}

We now investigate the positivity of semigroup $(\mathcal{T}_{P,L}(t))_{t\ge 0}$ by applying the developed formula of the resolvent $R(\lambda,\mathcal{A}_{P,L})$ and verifying its positivity for $\lambda$ big enough.

\begin{theo}\label{Main-result-3}
Let Assumption \ref{tildeA} be satisfied. In addition, we assume that the operators $M,L,P,$ $R(\lambda,A)$ and $D_\lambda$ are positive for $\lambda>s(A)$, and there exists $\la_0>s(A)$ such that
\begin{align}\label{Condition-positivity} r(MD_{\la_0})<1.
\end{align}
Then the semigroup $(\mathcal{T}_{P,L}(t))_{t\ge 0}$ is positive.
\end{theo}
\begin{proof}
We aim to prove that the operator $R(\la,\mathcal{A}_{P,L})$ is positive on a right half-line. According to the expression of its resolvent given in Theorem \ref{Main-result-2} and since, by \eqref{Condition-positivity}, $R(\la ,A^M)$ and $D_\la ^M$ are positive for all $\la >\la _0$, see
\cite[Theorem 3.2]{BBH-22}, and $R(\la,Q)$ is positive as well, it suffices to show that $(I-\Delta(\la))^{-1}$ is positive. Indeed,
for $\la>\la_0$,  we have
\begin{align*}(I-\Delta(\la))^{-1}&=(I-e_{\la}(D^M_\la L+ R(\la,A^M)P))^{-1}\cr &= I+e_\la \left(I-(D^M_\la L+ R(\la,A^M)P)e_\la\right)^{-1}(D^M_\la L+ R(\la,A^M)P).\end{align*}
From this equality we deduce that $(I-\Delta(\la))^{-1}$ is positive, if
\begin{equation}\label{eq:pos}
 \left(I-(D^M_\la L+ R(\la,A^M)P)e_\la\right)^{-1}\ge 0.\end{equation}
Observe that, by using the positivity of $M,\,D_\la$, \eqref{Condition-positivity}, the resolvent equation and  Theorem \ref{HMR15}, it is not difficult to see that $(D^M_\la)_{\la >\la _0}$ is a nonincreasing sequence. This implies that
$$Le_\la D^M_\la \le Le_\la D^M_{\la_0},\quad \text{ for all }\la >\la_0.$$
Thus, since $\lim_{\la \to\infty}\|Le_\la \|_{\mathcal{L}(X,U)}=0$, it follows that there exists a sufficiently large $\mu_0>\la_0$ such that
$$\|Le_\la D^M_\la\|_{\mathcal{L}(U)}\le \|Le_\la D^M_{\la_0}\|_{\mathcal{L}(U)}\le \|Le_\la \|_{\mathcal{L}(X,U)} \|D^M_{\la_0}\|_{\mathcal{L}(U,X)}<1,\quad  \text{ for all } \la \ge \mu_0.$$
Hence, for any $\la \ge\mu_0$,
\begin{eqnarray*}
0\le (I-D^M_\la Le_\la)^{-1} &=& I+D_\la ^M(I-Le_\la D^M_\la)^{-1}Le_\la \\
&=& I+D^M_\la \sum_{n=0}^\infty (Le_\la D^M_\la)^nLe_\la \\
&\le & (I-D^M_{\la_0} Le_{\la_0})^{-1}.
\end{eqnarray*}
From the above inequality, one obtains
$$
R(\la ,A^M)Pe_\la (I-D^M_\la Le_\la)^{-1}\le R(\la ,A^M)Pe_{\la_0}(I-D^M_{\la_0} Le_{\la_0})^{-1},\quad \la \ge \mu_0.$$
Therefore, there exists a sufficiently large $\mu_1>\mu_0$ such that for any $\la\ge \mu_1$, $$ \|R(\la,A^M)Pe_\la (I-D^M_\la Le_\la)^{-1}\|\le \|R(\la ,A^M)Pe_{\la_0}(I-D^M_{\la_0} Le_{\la_0})^{-1}\|<\frac{1}{2}.$$
On the other hand, observe that
$$ I-(D^M_\la L+ R(\la,A^M)P)e_\la= \left(I-R(\la,A^M)Pe_\la (I-D^M_\la Le_\la)^{-1}\right)(I-D^M_\la Le_\la)$$
for $\la \ge \mu_0$.
Thus, to achieve our initial goal (see \eqref{eq:pos}), it suffices to see that
\begin{align*}
&\left(I-(D^M_\la L+ R(\la,A^M)P)e_\la\right)^{-1} \\
&= (I-D^M_\la Le_\la)^{-1}\sum_{n=0}^\infty \left[R(\la,A^M)Pe_\la (I-D^M_\la Le_\la)^{-1}\right]^n
\ge  0
\end{align*}
for any $\la \ge \max\{\mu_0,\mu_1\}$.
This ends the proof.
\end{proof}

\section{Application: Flow in network with delays}\label{Section-Network}

The purpose of this section is to employ the theoretical findings derived in Sections \ref{sec:3} and \ref{sec:4} to prove the well-posedness and {the positivity of the solution to} network system with delays, as expressed in equation \eqref{del-f}. To achieve this objective, we will transform the system \eqref{del-f} into an abstract delay system that conforms to the structure of equation \eqref{BED}. Specifically, we consider the state space
\begin{align*}
X:=L^p([0,1],\C^m)\quad\text{with the norm}\quad \|f\|_X^p:=\sum_{k=1}^m \int^1_0 |f_k(s)|^p\, ds,
\end{align*}
the boundary space $U=\CC^m$ with the standard norm, and the phase space
\begin{align*}
{Y}:=L^p([-1,0],X)\quad\text{with the norm}\quad \|g\|_{{Y}}^p:= \int^0_{-1}\|g(\si)\|^p_{{X}}\,d\si.
\end{align*}
In the context of this study, we will make the assumption that the \emph{delay operators} can be represented as
\begin{align*}
P_{k}g_k=\int^0_{-1} d\mu_k(\theta)g_k(\theta) \quad\text{and}\quad \ell_k g_k=\int^0_{-1} d\nu_k(\theta)g_k(\theta).
\end{align*}
We define the operators $P\colon W^{1,p}([-1,0],X)\to X$ and $\ell\colon W^{1,p}([-1,0],X)\to \C^m$ as
\[P:={\rm diag}(P_k)\quad \text{and }\quad \ell:={\rm diag}(\ell_k).\]
Then, the  boundary condition {presented in the third line of  \eqref{del-f}}{ (see also  \cite[Proposition 2.1]{BDR}),}
can be reformulated as
\begin{align*}
f(1)=\B f(0)+\B \ell(g).
\end{align*}

We now define the operators on spaces X and Y, respectively.
\begin{align*}
&D(A_m):=W^{1,p}([0,1],\C^m),\quad A_m:=\frac{d}{dx},\cr
&D(Q_m):=W^{1,p}([-1,0],X),\quad Q_m:=\frac{d}{d\si}.
\end{align*}
By introducing the product spaces $\X:=X\times Y$ and $\mathcal{Z}:= D(A_m)\times D(Q_m)$, 
system \eqref{del-f} can be expressed as an abstract Cauchy problem, 
\begin{align*}
\begin{cases}
\dot{\calU}(t)=\A\; \calU(t),& t\ge 0,\\
\calU(0)=(\begin{smallmatrix}f\\g\end{smallmatrix})\in X\times{Y},
\end{cases}
\end{align*}
where $(\A,D(\A))$ is a linear operator on $\X=X\times{Y}$ defined by
\begin{align*}
D(\A)&:=\Big\{(\begin{smallmatrix}f\\g\end{smallmatrix})\in \mathcal{Z}: f(1)=\B f(0)+\B \ell(g),\;f=g(0)\Big\},\\
\A&:=\begin{pmatrix} A_m& P\\ 0& Q_m\end{pmatrix}.
\end{align*}
\begin{theo}\label{thm-network}
{Operator  $\A$ generates a positive strongly continuous semigroup $(\T(t))_{t\ge 0}$ on $\X$. Furthermore, the network system with delays  \eqref{del-f} is well-posed and the solutions $u=(u_j)_{1\le j\le m}\in X$ can be obtained from the formula
\begin{align*} \T(t)\left(\begin{smallmatrix}f\\g\end{smallmatrix}\right)=\begin{pmatrix}u(\cdot,t)\\ u(\cdot,t+\cdot)\end{pmatrix},\end{align*}
where $f\in X$, $g\in Y$ are given initial functions.
}
 \end{theo}
\begin{proof}
Let us define $L$ as the operator $\B \ell$, and consider the operators $G$ and $M$, defined as
\begin{eqnarray*}
& & G,\,M\colon W^{1,p}([0,1],\C^m)\to \C^m;\\
& & Gf=f(1),\,Mf=\B f(0).
\end{eqnarray*}
Using this notation, we can observe that the operator $\A$ is similar to the operator $\calA_{P,L}$ defined in Section \ref{sec:3}. {By Theorem \ref{Main-result-1}, it suffices} to verify Assumptions \ref{tildeA}. It is evident that $G: W^{1,p}([0,1],\C^m)\to \C^m$ is surjective, and $A=A_m$ with domain
\[D(A)=\{f\in W^{1,p}([0,1],\C^m):f(1)=0\}\]
 generates the $C_0$-semigroup $(T(t))_{t\ge 0}$ on $X$ given by
\[(T(t)f)(x)=
\begin{cases}
			f(x+t), & \text{if $x+t\le 1$,}\\
            0, & \text{otherwise.}
\end{cases}
\]
In addition, the Dirichlet operator $D_\lambda \colon \C^m\to \ker(\lambda -A_m)$ can be computed explicitly and is given by
$$(D_\lambda a)(x)=e^{\lambda(x-1)}a,\quad a\in \C^m,\,x\in [0,1],\,\lambda \in \C.$$
We choose {$B:=(\la-A_{-1})D_\la$} for $\la\in \rho(A)$. Additionally, we define
\begin{eqnarray*}
(\Phi_t u)(x)=
\begin{cases}
u(x+t-1), & \text{if $x+t\ge 1,$}\\
            0, & \text{otherwise.}
\end{cases}
\end{eqnarray*}
Thus, for any $t\ge 0$, $\Phi_t$ is a linear and bounded operator from $L^p(\R^+,\C^m)$ to $X$. Furthermore, it can be observed that the Laplace transform of $(t\mapsto \Phi_t u)$ is precisely $D_\la \hat{u}(\la)$, {where $\hat{u}$ denotes the Laplace transform of $u$}. Therefore, by the injectivity of the Laplace transform, {we have}
\begin{align*} \Phi_t u=\int^t_0 T_{-1}(t-s)Bu(s)ds\in X,\quad t\ge 0.\end{align*}
This implies that $B$ is an admissible control operator for $A$. Let us now consider the observation operator denoted as $C:=M_{|D(A)}$. For any $\alpha\in (0,1]$ and $f\in D(A)$, we can establish the following estimate:
\begin{eqnarray*}
\int_0^t |CT(s)f|_{\mathbb{C}^m}^pds &=& \int_0^t |\B f(s)|_{\mathbb{C}^m}^p\\
&\le & \|\B\|^p\|f\|_p^p.
\end{eqnarray*}
From this, we can conclude that $C$ is an admissible observation operator for $A$.  On the other hand,
{by the resolvent identity and equation $D_\la a = R(\la,A)Ba$, we obtain
\begin{align*}
(\gamma R(\gamma ,A)D_\lambda a)(x) &= \frac{\gamma}{\gamma -\lambda}\left(R(\la, A) - R(\gamma, A)\right) Ba \\
&=\frac{\gamma}{\gamma -\lambda}\left(e^{\la(x-1)}-e^{\gamma (x-1)}\right)a
\end{align*}
for any $\gamma >0$, $x\in [0,1]$, $a\in \C^m$, $\lambda \in \C.$} So,
$$M(\gamma R(\gamma ,A)D_\lambda a)=\frac{\gamma}{\gamma -\lambda}\left(e^{-\la}-e^{-\gamma}\right)\B a,\quad \gamma >0,\,a\in \C^m.$$
Taking the limit $\gamma \to +\infty$ we obtain {for the Yosida extension, see \eqref{eq:yosida},}
$${\rm Range}(D_\lambda)\subset D(C_\Lambda) \hbox{\ and }C_\Lambda D_\lambda a=e^{-\lambda}\B a,\quad \lambda \in \C ,\,a\in \C^m.$$
Thus, by Theorem \ref{regular}, it follows that
the triple $(A,B,C)$ generates a regular linear system on $X,\,\C^m$.
Finally, we note that
$$(\F_t^{A,B,C} u)(\tau)=
\begin{cases}
\B u(\tau -1), & \text{if $\tau \ge 1,$}\\
            0, & \text{otherwise}
\end{cases}
$$ on $[0,t]$. So, for $t_0<1$, operator {$I_{\C^m} -\F_{t_0}=I_{\C^m}$ is invertible.}

We have verified all the conditions in Assumption \ref{tildeA}.
{Moreover, since $MD_\la =e^{-\lambda}\mathbb{B}$ and $\B$ is by \eqref{eq:w} stochastic, we have $r(MD_\la) <1$ for all $\lambda >0$.
All the statements now follow from Theorem \ref{Main-result-1} and Theorem \ref{Main-result-3}.}
\end{proof}

\section*{Acknowledgement}
This research was initiated at the joint stay of the authors at the Mathematisches Forschungsinstitut Oberwolfach in 2015 supported through the program ”Research in Pairs”.
This article is based upon work from COST Action CA18232, supported by COST (European Cooperation in Science and Technology), www.cost.eu.
The second author acknowledges financial support from the Slovenian Research Agency (ARIS), Grant
No. P1-0222.
The third author is member of the Gruppo Nazionale per l'Analisi Matematica, la Probabilità e le loro Applicazioni (GNAMPA) of the Istituto Nazionale di Alta Matematica (INdAM). He is supported by the M.U.R. Research Project Prin 2022: D53D23005580006 ”Elliptic and
parabolic problems, heat kernel estimates and spectral theory”. 

The authors thank Said Hadd (Agadir) for many fruitful discussions and for his help in many parts of this research. 

This work does not have any conflicts of interest.

\end{document}